\documentclass[leqno,12pt]{amsart}
\usepackage{amssymb}
\usepackage{amsmath}
\usepackage{enumerate}
\usepackage{amsfonts}
\usepackage{hyperref}

\usepackage[headheight=18pt, top=25mm, bottom=25mm, left=25mm, right=25mm]{geometry}

\usepackage{tikz}
\definecolor{darkgreen}{RGB}{45, 119, 75}

\newtheorem{theorem}{Theorem}[section]
\newtheorem{corollary}[theorem]{Corollary}
\newtheorem{lemma}[theorem]{Lemma}
\newtheorem{proposition}[theorem]{Proposition}
\newtheorem{remark}[theorem]{Remark}

\numberwithin{equation}{section}

\hypersetup{
    colorlinks=true,
    linkcolor= blue,
    citecolor =cyan,
    urlcolor = teal,
}

\begin{document}

\title[Heat kernel estimates]{Upper and lower bounds for Dunkl heat kernel}

\subjclass[2010]{{primary: 44A20, 35K08, 33C52, 43A32, 39A70}}
\keywords{Rational Dunkl theory, heat kernel, root system}

\author[Jacek Dziubański]{Jacek Dziubański}
\author[Agnieszka Hejna]{Agnieszka Hejna}
\begin{abstract} On $\mathbb R^N$ equipped with a normalized root system $R$, a multiplicity function $k(\alpha) > 0$, and the associated measure 
 $$ 
dw(\mathbf x)=\prod_{\alpha\in R}|\langle \mathbf x,\alpha\rangle|^{k(\alpha)}\, d\mathbf x, 
$$ 
let $h_t(\mathbf x,\mathbf y)$ denote the heat kernel of the semigroup generated by the Dunkl Laplace operator $\Delta_k$. Let $d(\mathbf x,\mathbf y)=\min_{\sigma\in G} \| \mathbf x-\sigma(\mathbf y)\|$, where $G$ is the reflection group associated with $R$.   We derive the following upper and lower bounds for $h_t(\mathbf x,\mathbf y)$: for all $c_l>1/4$ and $0<c_u<1/4$ there are constants $C_l,C_u>0$ such that  
$$
 C_{l}w(B(\mathbf{x},\sqrt{t}))^{-1}e^{-c_{l}\frac{d(\mathbf{x},\mathbf{y})^2}{t}} \Lambda(\mathbf x,\mathbf y,t) \leq    h_t(\mathbf{x},\mathbf{y}) \leq  C_{u}w(B(\mathbf{x},\sqrt{t}))^{-1}e^{-c_{u}\frac{d(\mathbf{x},\mathbf{y})^2}{t}} \Lambda(\mathbf x,\mathbf y,t),
$$
where $\Lambda(\mathbf x,\mathbf y,t)$ can be expressed by means of  some   rational functions of  $\| \mathbf x-\sigma(\mathbf y)\|/\sqrt{t}$. An exact formula for $\Lambda(\mathbf x,\mathbf y,t)$ is provided. 
\end{abstract}

\address{Jacek Dziubański, Uniwersytet Wroc\l awski,
Instytut Matematyczny,
Pl. Grunwaldzki 2,
50-384 Wroc\l aw,
Poland}
\email{jdziuban@math.uni.wroc.pl}

\address{Agnieszka Hejna, Uniwersytet Wroc\l awski,
Instytut Matematyczny,
Pl. Grunwaldzki 2,
50-384 Wroc\l aw,
Poland}
\email{hejna@math.uni.wroc.pl}

\thanks{
Research supported by the National Science Centre, Poland (Narodowe Centrum Nauki), Grant 2017/25/B/ST1/00599}

\maketitle

\section{Introduction and statement of the results}\label{Sec:Intro}

{
On the Euclidean space  $\mathbb R^N$ equipped with a normalized root system $R$ and a multiplicity function $k(\alpha) > 0$, let $\Delta_k$ denote the Dunkl Laplace operator (see Section \ref{sec:preliminaries}). Let $dw(\mathbf x)=w(\mathbf x)\, d\mathbf x$ be the associated measure, where 
 \begin{equation}\label{eq:measure}
w(\mathbf x)=\prod_{\alpha\in R}|\langle \mathbf x,\alpha\rangle|^{k(\alpha)}.
\end{equation}
It is well-known that $\Delta_k$ generates a semigroup $\{e^{t\Delta_k}\}_{t \geq 0}$ of linear operators on $L^2(dw)$ which has the form 
\begin{align*}
    e^{t\Delta_k}f(\mathbf x)=\int_{\mathbb{R}^N} h_t(\mathbf x,\mathbf y)f(\mathbf y)\, dw(\mathbf y),
\end{align*}
where $0<h_t(\mathbf x,\mathbf y)$ is a smooth function called the Dunkl heat kernel.

The main goal of this paper is to prove upper and lower bounds for $h_t(\mathbf x,\mathbf y)$. In order to state the result we need to introduce some notation. 

For $\alpha\in R$, let 
\begin{equation}\label{reflection}\sigma_\alpha (\mathbf x)=\mathbf x-2\frac{\langle \mathbf x,\alpha\rangle}{\|\alpha\|^2} \alpha
\end{equation}
stand for  the reflection with respect to the subspace  perpendicular to $\alpha$. Let $G$ denote the Coxeter (reflection)  group generated by the reflections $\sigma_\alpha$, $\alpha\in R$. We define the distance of the orbit of $\mathbf x$ to the orbit of $\mathbf y$ by 
$$ d(\mathbf x,\mathbf y)=\min \{ \| \mathbf x-\sigma (\mathbf y)\|: \sigma \in G\}.$$ 
Obviously,
\begin{align*}
    d(\mathbf x,\mathbf y)=d(\mathbf x,\sigma (\mathbf y)) \quad \text{for all } \mathbf x,\mathbf y\in \mathbb R^N \ \text{\rm and } \sigma \in G.
\end{align*}
It is well known that $d(\mathbf x,\mathbf y)=\|\mathbf x-\sigma (\mathbf y)\|$ if and only if $\mathbf x$ and $\sigma(\mathbf y)$ belong to the same (closed) Weyl chamber (see \cite[Chapter VII, proof of Theorem 2.12]{Helgason}). Let
\begin{align*}
    B(\mathbf x,r)=\{ \mathbf x'\in \mathbb R^N: \| \mathbf x-\mathbf x'\|\leq r\}    
\end{align*}
stands for the Euclidean ball centered at $\mathbf x$ and radius $r$. We denote by $w(B(\mathbf{x},r))$ the $dw$-volume of the ball $B(\mathbf{x},r)$.

For a finite  sequence $\boldsymbol \alpha  =(\alpha_1,\alpha_2,...,\alpha_m)$ of elements of $R$, $\mathbf x,\mathbf y\in \mathbb R^N$ and $t>0$,  let 

\begin{equation}
    \ell(\mathbf{\boldsymbol \alpha}):=m 
\end{equation}
be the length of $\boldsymbol \alpha$, 
\begin{equation}
    \sigma_{\boldsymbol \alpha}:=\sigma_{\alpha_m}\circ \sigma_{\alpha_{m-1}}\circ ...\circ\sigma_{\alpha_1}, 
\end{equation}
and 
 \begin{equation}\label{eq:rho}\begin{split}
    &\rho_{\boldsymbol \alpha}(\mathbf{x},\mathbf{y},t)\\:=&\left(1+\frac{\|\mathbf{x}-\mathbf{y}\|}{\sqrt{t}}\right)^{-2}\left(1+\frac{\|\mathbf{x}-\sigma_{\alpha_1}(\mathbf{y})\|}{\sqrt{t}}\right)^{-2}\left(1+\frac{\|\mathbf{x}-\sigma_{\alpha_2} \circ \sigma_{\alpha_1}(\mathbf{y})\|}{\sqrt{t}}\right)^{-2}\cdot\ldots \cdot \\&\left(1+\frac{\|\mathbf{x}-\sigma_{\alpha_{m-1}} \circ \ldots \circ \sigma_{\alpha_1}(\mathbf{y})\|}{\sqrt{t}}\right)^{-2}.
    \end{split}
\end{equation}

For $\mathbf x,\mathbf y\in\mathbb R^N$, let 
    $ n (\mathbf x,\mathbf y)=0$ if $  d(\mathbf x,\mathbf y)=\| \mathbf x-\mathbf y\|$ and  
    
    \begin{equation}\label{eq:n}
        n(\mathbf x,\mathbf y) = 
    \min\{m\in\mathbb Z: d(\mathbf x,\mathbf y)=\| \mathbf x-\alpha_{m}\circ \ldots \circ \sigma_{\alpha_2}\circ\sigma_{\alpha_1}(\mathbf y)\|,\quad \alpha_j\in R\}     
    \end{equation}
otherwise. In other words, $n(\mathbf x,\mathbf y)$ is the smallest number of reflections $\sigma_\alpha$ which are needed to move $\mathbf y$ to the (closed) Weyl chamber of $\mathbf x$ (see Subsection~\ref{sec:chamber}).  
We also allow  $\boldsymbol{\alpha}$ to be the empty sequence, denoted by $\boldsymbol{\alpha} =\emptyset$. Then for $\boldsymbol{\alpha}=\emptyset$, we set:   $\sigma_{\boldsymbol{\alpha}}=I$ (the identity operator), $\ell(\boldsymbol{\alpha})=0$, and $\rho_{\boldsymbol{\alpha}}(\mathbf{x},\mathbf{y},t)=1$ for all $\mathbf{x},\mathbf{y} \in \mathbb{R}^N$ and $t>0$. 
   
We say that  a finite  sequence $\boldsymbol \alpha  =(\alpha_1,\alpha_2,...,\alpha_m)$ of positive roots is {\it admissible for the pair}  $(\mathbf x,\mathbf y)\in\mathbb R^N\times\mathbb R^N$  if $n(\mathbf{x},\sigma_{\boldsymbol{\alpha}}(\mathbf{y}))=0$. In other words, the composition $\sigma_{\alpha_m}\circ\sigma_{\alpha_{m-1}}\circ \ldots\circ \sigma_{\alpha_1}$ of the reflections $\sigma_{\alpha_j}$ maps $\mathbf y$ to the Weyl chamber of $\mathbf x$. 

The set of the all admissible sequences $\boldsymbol \alpha$ for the pair  $(\mathbf x,\mathbf y)$  will be denoted by $\mathcal A(\mathbf x,\mathbf y)$.  
Note that if $n(\mathbf x,\mathbf y)=0$, then $\boldsymbol{\alpha}=\emptyset \in \mathcal A(\mathbf{x},\mathbf{y})$.

Let us define
\begin{equation}
    \Lambda (\mathbf x,\mathbf y,t):=\sum_{\boldsymbol \alpha \in \mathcal{A}(\mathbf{x},\mathbf{y}), \;\ell (\boldsymbol \alpha) \leq 2|G|} \rho_{\boldsymbol \alpha}(\mathbf x,\mathbf y,t).
\end{equation}

Note that for any $c>1$ and for all $\mathbf{x},\mathbf{y} \in \mathbb{R}^N$ and $t>0$ we have
\begin{equation}\label{eq:Lambda_2t}
c^{-2|G|} \Lambda(\mathbf x,\mathbf y,ct)\leq  \Lambda(\mathbf x,\mathbf y,t)\leq \Lambda(\mathbf x,\mathbf y,ct).
\end{equation}

We are now in a position to state our main result about upper and lower bounds for the Dunkl heat  kernel which are given by means of $w$-volumes of the Euclidean balls, the function   $\Lambda(\mathbf{x},\mathbf{y},t)$,  and $d(\mathbf x,\mathbf y)$. 

\begin{theorem}\label{teo:1}
Assume that $0<c_{u}<1/4$ and $c_l>1/4$. There are constants $C_{u},C_{l}>0$ such that for all $\mathbf{x},\mathbf{y} \in \mathbb{R}^N$ and $t>0$ we have 
\begin{equation}\label{eq:main_lower}
 C_{l}w(B(\mathbf{x},\sqrt{t}))^{-1}e^{-c_{l}\frac{d(\mathbf{x},\mathbf{y})^2}{t}} \Lambda(\mathbf x,\mathbf y,t) \leq    h_t(\mathbf{x},\mathbf{y}), 
\end{equation}
\begin{equation}\label{eq:main_claim}
    h_t(\mathbf{x},\mathbf{y}) \leq C_{u}w(B(\mathbf{x},\sqrt{t}))^{-1}e^{-c_{u}\frac{d(\mathbf{x},\mathbf{y})^2}{t}} \Lambda(\mathbf x,\mathbf y,t).
\end{equation}
\end{theorem}

Let us remark that this way of expressing estimates of the heat kernel is convenient in handling real harmonic analysis problems, because it allows us to apply methods from analysis on spaces of homogeneous type in the sense of Coifman and Weiss.

 The proof of the theorem is based on an iteration procedure. In order to illustrate the method we start by proving upper and lower bounds for $h_t(\mathbf x,\mathbf y)$ in  the case where the root system is associated with symmetries of a regular $m$-sided polygon in $\mathbb R^2$, e.g. when $G$ is the  dihedral group. In this case the formulation of the estimates and they proofs are much simpler. 

\begin{theorem}\label{th:dihedral}
Assume that $G$ is a group of symmetries of a regular $m$-sided polygon  in $\mathbb R^2$ centered at the origin and let $R$ be the associated root system. Set 
$$ \Lambda_D(\mathbf x,\mathbf y,t):=
\begin{cases}1 \quad &\text{ if }  n(\mathbf x,\mathbf y)=0,\\
\Big(1+\frac{\| \mathbf x-\mathbf y\|}{\sqrt{t}}\Big)^{-2} \quad &\text{ if }  n(\mathbf x,\mathbf y)=1,\\
\Big(1+\frac{\| \mathbf x-\mathbf y\|}{\sqrt{t}}\Big)^{-2} \sum_{\alpha\in {R_{+}}}\Big(1+\frac{\| \mathbf x-\mathbf \sigma_{\alpha}(\mathbf y)\|}{\sqrt{t}}\Big)^{-2} \quad & \text{ if }  n(\mathbf x,\mathbf y)=2. \end{cases} $$ 
Let $0<c_{u}<1/4$ and $c_l>1/4$. There are constants $C_{u},C_{l}>0$ such that for all $\mathbf{x},\mathbf{y} \in \mathbb{R}^N$ and $t>0$ we have 
    \begin{equation}\label{eq:main_lower_D}
 C_{l}w(B(\mathbf{x},\sqrt{t}))^{-1}e^{-c_{l}\frac{d(\mathbf{x},\mathbf{y})^2}{t}} \Lambda_D(\mathbf x,\mathbf y,t) \leq    h_t(\mathbf{x},\mathbf{y}), 
\end{equation}
\begin{equation}\label{eq:main_claim_D}
    h_t(\mathbf{x},\mathbf{y}) \leq C_{u}w(B(\mathbf{x},\sqrt{t}))^{-1}e^{-c_{u}\frac{d(\mathbf{x},\mathbf{y})^2}{t}} \Lambda_D(\mathbf x,\mathbf y,t).
\end{equation}
\end{theorem}

\ 

Theorems~\ref{teo:1} and~\ref{th:dihedral} can be consider as improvements of the following estimates
\begin{align}\label{eq:heat_JFAA}
   {C^{-1}}w(B(\mathbf{x},\sqrt{t}))^{-1} e^{-c^{-1}\|\mathbf{x}-\mathbf{y}\|^2/t}\leq h_t(\mathbf{x},\mathbf{y}) \leq {C} w(B(\mathbf{x},\sqrt{t}))^{-1}  e^{-{c}d(\mathbf{x},\mathbf{y})^2/t}
\end{align}
obtained in~{\cite[Theorems 3.1 and 4.4]{ADzH}}  (see Section~\ref{sec:auxiliary} for more details).
}

\section{Preliminaries and notation}\label{sec:preliminaries}

\subsection{Basic definitions of the Dunkl theory}

In this section we present basic facts concerning the theory of the Dunkl operators.   For more details we refer the reader to~\cite{Dunkl},~\cite{Roesler3}, and~\cite{Roesler-Voit}. 

We consider the Euclidean space $\mathbb R^N$ with the scalar product $\langle \mathbf{x},\mathbf y\rangle=\sum_{j=1}^N x_jy_j
$, where $\mathbf x=(x_1,...,x_N)$, $\mathbf y=(y_1,...,y_N)$, and the norm $\| \mathbf x\|^2=\langle \mathbf x,\mathbf x\rangle$.

A {\it normalized root system}  in $\mathbb R^N$ is a finite set  $R\subset \mathbb R^N\setminus\{0\}$ such that $R \cap \alpha \mathbb{R} = \{\pm \alpha\}$,  $\sigma_\alpha (R)=R$, and $\|\alpha\|=\sqrt{2}$ for all $\alpha\in R$, where $\sigma_\alpha$ is defined by \eqref{reflection}. Each root system
can be written as a disjoint union $R = R_{+} \cup -R_{+}$, where $R_{+}$, $-R_{+}$ are separated by a hyperplane through the origin. Such a set $R_{+}$ is called a \textit{positive subsystem}. Its choice is not unique. In this paper, we will work with fixed root system $R_{+}$.

The finite group $G$ generated by the reflections $\sigma_{\alpha}$, $\alpha \in R$ is called the {\it Coxeter group} ({\it reflection group}) of the root system. Clearly, $|G|\geq |R|$.

A~{\textit{multiplicity function}} is a $G$-invariant function $k:R\to\mathbb C$ which will be fixed and $> 0$  throughout this paper.  

Let $\mathbf{N}=N+\sum_{\alpha \in R}k(\alpha)$. Then, 
\begin{align*} w(B(t\mathbf x, tr))=t^{\mathbf{N}}w(B(\mathbf x,r)) \ \ \text{\rm for all } \mathbf x\in\mathbb R^N, \ t,r>0,
\end{align*}
where $w$ is the associated measure defined in \eqref{eq:measure}. 
Observe that there is a constant $C>0$ such that for all $\mathbf{x} \in \mathbb{R}^N$ and $r>0$ we have
\begin{equation}\label{eq:balls_asymp}
C^{-1}w(B(\mathbf x,r))\leq  r^{N}\prod_{\alpha \in R} (|\langle \mathbf x,\alpha\rangle |+r)^{k(\alpha)}\leq C w(B(\mathbf x,r)),
\end{equation}
so $dw(\mathbf x)$ is doubling, that is, there is a constant $C>0$ such that
\begin{equation}\label{eq:doubling} w(B(\mathbf x,2r))\leq C w(B(\mathbf x,r)) \ \ \text{ for all } \mathbf x\in\mathbb R^N, \ r>0.
\end{equation}

For $\xi \in \mathbb{R}^N$, the {\it Dunkl operators} $T_\xi$  are the following $k$-deformations of the directional derivatives $\partial_\xi$ by   difference operators:
\begin{align*}
     T_\xi f(\mathbf x)= \partial_\xi f(\mathbf x) + \sum_{\alpha\in R} \frac{k(\alpha)}{2}\langle\alpha ,\xi\rangle\frac{f(\mathbf x)-f(\sigma_\alpha(\mathbf{x}))}{\langle \alpha,\mathbf x\rangle}.
\end{align*}
The Dunkl operators $T_{\xi}$, which were introduced in~\cite{Dunkl}, commute and are skew-symmetric with respect to the $G$-invariant measure $dw$.

 Let us denote $T_j=T_{e_j}$, where $\{e_j\}_{1 \leq j \leq N}$ is a canonical orthonormal  basis of $\mathbb{R}^N$.

For fixed $\mathbf y\in\mathbb R^N$ the {\it Dunkl kernel} $E(\mathbf x,\mathbf y)$ is the unique analytic solution to the system
\begin{align*}
    T_\xi f=\langle \xi,\mathbf y\rangle f, \ \ f(0)=1.
\end{align*}
The function $E(\mathbf x ,\mathbf y)$, which generalizes the exponential  function $e^{\langle \mathbf x,\mathbf y\rangle}$, has a  unique extension to a holomorphic function on $\mathbb C^N\times \mathbb C^N$.

\subsection{Dunkl Laplacian and Dunkl heat semigroup} The {\it Dunkl Laplacian} associated with $R$ and $k$  is the differential-difference operator $\Delta_k=\sum_{j=1}^N T_{j}^2$, which  acts on $C^2(\mathbb{R}^N)$-functions by\begin{align*}
    \Delta_k f(\mathbf x)=\Delta_{\rm eucl} f(\mathbf x)+\sum_{\alpha\in R} k(\alpha) \delta_\alpha f(\mathbf x),
\end{align*}
\begin{align*}
    \delta_\alpha f(\mathbf x)=\frac{\partial_\alpha f(\mathbf x)}{\langle \alpha , \mathbf x\rangle} - \frac{\|\alpha\|^2}{2} \frac{f(\mathbf x)-f(\sigma_\alpha (\mathbf x))}{\langle \alpha, \mathbf x\rangle^2}.
\end{align*}
The operator $\Delta_k$ is essentially self-adjoint on $L^2(dw)$ (see for instance \cite[Theorem\;3.1]{AH}) and generates a semigroup $H_t$  of linear self-adjoint contractions on $L^2(dw)$. The semigroup has the form
  \begin{align*}
  H_t f(\mathbf x)=\int_{\mathbb R^N} h_t(\mathbf x,\mathbf y)f(\mathbf y)\, dw(\mathbf y),
  \end{align*}
  where the heat kernel
  \begin{equation}\label{eq:heat_formula}
      h_t(\mathbf x,\mathbf y)={\boldsymbol c}_k^{-1} (2t)^{-\mathbf{N}/2}E\Big(\frac{\mathbf x}{\sqrt{2t}}, \frac{\mathbf y}{\sqrt{2t}}\Big)e^{-(\| \mathbf x\|^2+\|\mathbf y\|^2)\slash (4t)}
  \end{equation}
  is a $C^\infty$-function of all variables $\mathbf x,\mathbf y \in \mathbb{R}^N$, $t>0$, and satisfies \begin{align*} 0<h_t(\mathbf x,\mathbf y)=h_t(\mathbf y,\mathbf x).
  \end{align*}
Here and subsequently, 
\begin{align*}
    {\boldsymbol c}_k=\int_{\mathbb R^N} e^{-\| \mathbf x\|^2/2}\, dw(\mathbf x).
\end{align*}

The following specific formula for the Dunkl heat kernel was obtained by R\"osler \cite{Roesler2003}:
\begin{equation}\label{heat:Rosler}
h_t(\mathbf x,\mathbf y)={\boldsymbol c}_k^{-1}2^{-\mathbf{N}/2} t^{-\mathbf{N}/2}\int_{\mathbb R^N}  \exp(-A(\mathbf x, \mathbf y, \eta)^2/4t)\,d\mu_{\mathbf{x}}(\eta)\text{ for all }\mathbf{x},\mathbf{y}\in\mathbb{R}^N, 
 t>0.
\end{equation}
Here
\begin{equation}\label{eq:A}
A(\mathbf{x},\mathbf{y},\eta)=\sqrt{{\|}\mathbf{x}{\|}^2+{\|}\mathbf{y}{\|}^2-2\langle \mathbf{y},\eta\rangle}=\sqrt{{\|}\mathbf{x}{\|}^2-{\|}\eta{\|}^2+{\|}\mathbf{y}-\eta{\|}^2}
\end{equation}
and $\mu_{\mathbf{x}}$ is a probability measure,
which is supported in the convex hull  $\operatorname{conv}\mathcal{O}(\mathbf{x})$ of the orbit  $\mathcal O(\mathbf x) =\{\sigma(\mathbf x): \sigma \in G\}$. 

One can easily check that 
\begin{equation}\label{eq:d_A}
    d(\mathbf x,\mathbf y)\leq A(\mathbf{x},\mathbf{y},\eta)\quad \text{for all } \eta \in \operatorname{conv}\mathcal{O}(\mathbf{x}). \end{equation}

{
\subsection{Weyl chambers and its properties}\label{sec:chamber}

The closures of connected components of
\begin{equation*}
    \{\mathbf{x} \in \mathbb{R}^{N}\;:\; \langle \mathbf{x},\alpha\rangle \neq 0 \text{ for all }\alpha \in R\}
\end{equation*}
are called (closed) \textit{Weyl chambers}. Below we present some properties of the reflections and the Weyl chambers, which will be used in next sections.

\begin{lemma} Fix $\mathbf x,\mathbf y\in\mathbb R^N$ and $\sigma \in G$. Then  $d(\mathbf x,\mathbf y)=\| \mathbf{x}-\sigma(\mathbf y)\|$ if and only if  $\sigma(\mathbf y)$ and $\mathbf x$ belong to the same Weyl chamber. 
\end{lemma}
\begin{proof}
See \cite[Chapter VII, proof of Theorem 2.12]{Helgason}. 
\end{proof}

\begin{lemma}\label{lem:alpha}
  Let $\mathbf{x},\mathbf{y} \in \mathbb{R}^N$ and assume that $n(\mathbf x,\mathbf y)\geq 1$. Then there is $\alpha\in R$ such that 
  \begin{equation}
      \| \mathbf x-\mathbf y\| >\| \mathbf{x}-\sigma_\alpha(\mathbf y)\|.
  \end{equation}
  \end{lemma}  
  \begin{proof}
If $ \| \mathbf x-\mathbf y\| \leq  \| \mathbf{x}-\sigma_\alpha(\mathbf y)\|$ for every $\alpha\in R$, then $\mathbf x$ and $\mathbf y$ are situated in the same (closed)  side of the hyperplane  $\alpha^{\perp}$ (\cite[Chapter VII, proof of Theorem 2.12]{Helgason}). Thus $\mathbf x$ and $\mathbf y$ belong to the same Weyl chamber, hence $n(\mathbf x,\mathbf y)=0$. 
  \end{proof}
  
  \begin{corollary}\label{coro:ciag_skracajacy}
  For any $\mathbf{x},\mathbf{y} \in \mathbb{R}^N$ such that $n(\mathbf{x},\mathbf{y})>0$ there are: $1 \leq m \leq |G|$ and $\boldsymbol{\alpha}=(\alpha_1,\alpha_2,\ldots,\alpha_m)$ such that
  \begin{equation}\label{eq:ciag_skracajacy}
      \|\mathbf{x}-\mathbf{y}\|>\|\mathbf{x}-\sigma_{\alpha_1}(\mathbf{y})\|>\|\mathbf{x}-\sigma_{\alpha_{2}} \circ \sigma_{\alpha_{1}}(\mathbf{y})\|>\ldots>\|\mathbf{x}-\sigma_{\alpha_{m}} \circ \sigma_{\alpha_{m-1}} \circ \ldots \circ \sigma_{\alpha_{1}}(\mathbf{y})\|
  \end{equation}
  and
  \begin{equation}\label{eq:skracajacy_koniec}
      n(\mathbf{x},\sigma_{\boldsymbol{\alpha}} (\mathbf{y}))=0.
  \end{equation}
  \end{corollary}
}
\section{Auxiliary estimates for the heat kernel}\label{sec:auxiliary} 

In the present section we establish auxiliary estimates for the heat kernel which will be used for proving Theorems \ref{teo:1} and \ref{th:dihedral}. Our starting point is the following proposition which is an improvement of the estimates \eqref{eq:heat_JFAA}.

 \begin{proposition}\label{propo:heat}
For any constants ${c_{\rm lower}}>1/4$ and $  0<{c_{\rm upper}}<1/4$ there are  constants ${C_{\rm lower}},{C_{\rm upper}}$, such that for all $\mathbf{x},\mathbf{y} \in \mathbb{R}^N$ and $t>0$ we have

\begin{equation}\label{upl} {C_{\rm lower}}w(B(\mathbf{x},\sqrt{t}))^{-1} e^{-{c_{\rm lower}}\|\mathbf{x}-\mathbf{y}\|^2/t}\leq h_t(\mathbf{x},\mathbf{y}) \leq {C_{\rm upper}} w(B(\mathbf{x},\sqrt{t}))^{-1}  e^{-{c_{\rm upper}}d(\mathbf{x},\mathbf{y})^2/t}.
\end{equation}
\end{proposition}
 \begin{proof}
To see the upper bound with any constant $0<c_{\rm upper}<1/4$ which is close to $1/4$, we apply \eqref{heat:Rosler} together with \eqref{eq:d_A} and get 
 \begin{equation*}
     \begin{split}
         h_t(\mathbf x,\mathbf y)&\leq {\boldsymbol c}_k^{-1}  (2t)^{-\mathbf{N}/2}\int_{\mathbb R^N}  \exp\Big(-\frac{(1-4c_{\rm upper})A(\mathbf x, \mathbf y, \eta)^2}{4t}\Big) \exp(-c_{\rm upper}d(\mathbf x,\mathbf y)^2/t)\,d\mu_{\mathbf{x}}(\eta)\\
         &=(1-4c_{\rm upper})^{-\mathbf N/2}h_{t/(1-4c_{\rm upper})}(\mathbf x,\mathbf y)\exp(-c_{\rm upper} d(\mathbf x,\mathbf y)^2/t)\\
         &\leq C_{\rm upper} w(B(\mathbf x,\sqrt{t}))^{-1} \exp(-c_{\rm upper} d(\mathbf x,\mathbf y)^2/t),
     \end{split}
 \end{equation*}
 where in the last inequality we have used the second inequality in \eqref{eq:heat_JFAA} and the doubling property \eqref{eq:doubling}. 
 
 The lower bound in \eqref{upl} with any constant $c_{\rm lower}>1/4$ is Corollary 2.3 of Jiu and Li \cite{JiuLi}. 
 \end{proof}
  We now turn to deriving estimates for the heat kernel which will be used for an iteration procedure.

\begin{proposition}\label{propo:basic}
Let ${c_{\rm lower}},{c_{\rm upper}}$ be the constants from Proposition~\ref{propo:heat} and let $c_1<c_{\rm upper}$. There is $C_1 \geq 1$ such that for all $\mathbf{x},\mathbf{y} \in \mathbb{R}^N$ and $t>0$ we have
\begin{equation}\label{eq:lower_1}
    C_1^{-1}\Bigg( w(B(\mathbf{x},\sqrt{t}))^{-1}e^{-{c_{\rm lower}}\frac{\|\mathbf{x}-\mathbf{y}\|^2}{t}}+\left(1+\frac{\|\mathbf{x}-\mathbf{y}\|}{\sqrt{t}}\right)^{-2}\sum_{\alpha \in R}h_{t}(\mathbf{x},\sigma_{\alpha}(\mathbf{y}))\Bigg) \leq h_t(\mathbf{x},\mathbf{y}),
\end{equation}
\begin{equation}\label{eq:upper_1}
    h_t(\mathbf{x},\mathbf{y}) \leq C_1\Bigg(w(B(\mathbf{x},\sqrt{t}))^{-1}e^{-c_{1}\frac{\|\mathbf{x}-\mathbf{y}\|^2}{t}}+\left(1+\frac{\|\mathbf{x}-\mathbf{y}\|}{\sqrt{t}}\right)^{-2}\sum_{\alpha \in R}h_{t}(\mathbf{x},\sigma_{\alpha}(\mathbf{y}))\Bigg).
\end{equation}
\end{proposition}

\begin{proof}
 The following formula was proved in~\cite[formula (3.5)]{DzH1}: for all $\mathbf{x},\mathbf{y} \in \mathbb{R}^N$ and $t>0$, we have
\begin{equation}\label{eq:studia}\partial_t h_t(\mathbf{x},\mathbf{y})=\frac{\| \mathbf{x}-\mathbf{y}\|^2}{(2t)^2}h_t(\mathbf{x},\mathbf{y})-\frac{N}{2t}h_t(\mathbf{x},\mathbf{y})-\frac{1}{2t}\sum_{\alpha\in R }k(\alpha)h_t(\mathbf{x},\sigma_\alpha(\mathbf{y})).
\end{equation}
On the other hand, by  \eqref{heat:Rosler}, 
\begin{equation}\label{eq:i1i2}
\begin{split}\partial_th_t(\mathbf{x},\mathbf{y})&=-\frac{\mathbf N}{2}t^{-1}h_t(\mathbf x,\mathbf y)+{\boldsymbol c}_k^{-1}2^{-\mathbf{N}/2}t^{-1}t^{-\mathbf N/2} \int_{\mathbb{R}^N} \frac{A(\mathbf x,\mathbf y,\eta)^2}{4t}e^{-A(\mathbf x,\mathbf y,\eta)^2/{4t}}\, d\mu_{\mathbf x}(\eta)\\
&=:I_1(t,\mathbf x,\mathbf y)+ I_2(t,\mathbf x,\mathbf y).
\end{split}
\end{equation}
Combining ~\eqref{eq:studia} with \eqref{eq:i1i2} we get   
\begin{equation}\label{eq:basic}
    \begin{split}
        \Big(2\mathbf N-2N+\frac{\| \mathbf x-\mathbf y\|^2}{t}\Big)h_t(\mathbf x,\mathbf y)&=4t I_2(t,\mathbf x,\mathbf y)+2\sum_{\alpha\in R} k(\alpha)h_t(\mathbf{x},\sigma_\alpha(\mathbf y)).
    \end{split}
\end{equation}
Note that  $I_2(t,\mathbf{x},\mathbf{y})>0$, and, thanks to our assumption on $k(\alpha)$,  $2\mathbf N-2N+1> 1$, so 
\begin{equation}\label{eq:h_t_k}
    \begin{split}
       h_t(\mathbf x,\mathbf y)
        &\geq 2(2\mathbf N-2N+1)^{-1} \Big(1+\frac{\| \mathbf x-\mathbf y\|^2}{t}\Big)^{-1}\sum_{\alpha\in R} k(\alpha)h_t(\mathbf{x},\sigma_\alpha(\mathbf y)).\\
    \end{split}
\end{equation}
Now, taking the arithmetic mean of the lower bound in \eqref{upl}   with \eqref{eq:h_t_k}   we obtain \eqref{eq:lower_1}, 
since there is a constant $C>0$ such that for all $\mathbf{x},\mathbf{y} \in \mathbb{R}^N$ and $t>0$ we have
 \begin{align*}
     C^{-1}\left(1+\frac{\|\mathbf{x}-\mathbf{y}\|^2}{t}\right)^{-1} \leq \left(1+\frac{\|\mathbf{x}-\mathbf{y}\|}{\sqrt{t}}\right)^{-2} \leq C\left(1+\frac{\|\mathbf{x}-\mathbf{y}\|^2}{t}\right)^{-1}.
 \end{align*}
% \begin{equation}
%      \Big(1+\frac{\| \mathbf x-\mathbf y\|^2}{t}\Big)h_t(\mathbf x,\mathbf y)
%         \geq c\sum_{\alpha\in R} k(\alpha)h_t(\sigma_\alpha(\mathbf x),\mathbf y),\\
% \end{equation}
% Since $h_t(\mathbf x,\mathbf y)\geq cw(B(\mathbf{x},\sqrt{t}))^{-1} e^{-C\| \mathbf x-\mathbf y\|^2/t},$ we get the lower bound
% \begin{equation}
%      h_t(\mathbf x,\mathbf y)
%         \geq cw(B(\mathbf{x},\sqrt{t}))^{-1} e^{-C\| \mathbf x-\mathbf y\|^2/t}+c \Big(1+\frac{\| \mathbf x-\mathbf y\|^2}{t}\Big)^{-1} \sum_{\alpha\in R} k(\alpha)h_t(\sigma_\alpha(\mathbf x),\mathbf y).
% \end{equation}

In order to prove~\eqref{eq:upper_1}, set $\varepsilon=(c_{\rm upper}-c_1)/(2c_{\rm upper})$. Clearly, by the assumption $c_1<c_{\rm upper}$,  we have $\varepsilon>0$. To obtain~\eqref{eq:upper_1}, we split the integral for $tI_2(t,\mathbf{x},\mathbf{y})$ as follows:
\begin{equation}
    \begin{split}
        tI_2(t,\mathbf x,\mathbf y)&={\boldsymbol c}_k^{-1}2^{-\mathbf{N}/2} t^{\mathbf N/2} \int_{A(\mathbf x,\mathbf y,\eta)^2\leq (1-\varepsilon)\| \mathbf x-\mathbf y\|^2} ...+{\boldsymbol c}_k^{-1}2^{-\mathbf{N}/2} t^{\mathbf N/2} \int_{A(\mathbf x,\mathbf y,\eta)^2> (1-\varepsilon)\| \mathbf x-\mathbf y\|^2}...\\
        &=:tI_{2,1}(t,\mathbf x,\mathbf y)+tI_{2,2}(t,\mathbf x,\mathbf y).  
    \end{split}
\end{equation}
Clearly,  
\begin{equation}\label{eq:I21}
    tI_{2,1}(t,\mathbf x,\mathbf y)\leq {\boldsymbol c}_k^{-1}2^{-\mathbf{N}/2}(1-\varepsilon)t^{-\mathbf N/2}\int_{\mathbb{R}^N}\frac{\|\mathbf x-\mathbf y\|^2}{4t}e^{-A(\mathbf x,\mathbf y,\eta)^2/{4t}}\, d\mu_{\mathbf x}(\eta)\leq (1-\varepsilon)  \frac{\| \mathbf x-\mathbf y\|^2}{4t} h_t(\mathbf x,\mathbf y).
\end{equation}
In order to estimate $I_{2,2}(t,\mathbf{x},\mathbf{y})$, note that there are $C_\varepsilon,C_{\varepsilon}',C_{\varepsilon}''>0$ such that 
\begin{equation}\label{eq:I22}
    \begin{split}
        tI_{2,2}(t,\mathbf x,\mathbf y)&\leq C_{\varepsilon}t^{-\mathbf N/2}\int_{A(\mathbf x,\mathbf y,\eta)^2> (1-\varepsilon)\| \mathbf x-\mathbf y\|^2}  e^{-(1-2\varepsilon)A(\mathbf x,\mathbf y,\eta)^2/(4t(1-\varepsilon))}e^{-\varepsilon A(\mathbf x,\mathbf y,\eta)^2/(8t(1-\varepsilon))}\, d\mu_{\mathbf x} (\eta)\\
       & \leq C_\varepsilon e^{-(1-2\varepsilon)\| \mathbf x-\mathbf y\|^2/4t}t^{-\mathbf N/2}\int_{\mathbb R^N} e^{-\varepsilon A(\mathbf x,\mathbf y,\eta)^2/(8t(1-\varepsilon))} \, d\mu_{\mathbf x}(\eta)\\
       &= C_\varepsilon'  e^{-(1-2\varepsilon)\| \mathbf x-\mathbf y\|^2/4t} h_{2(1-\varepsilon)t/\varepsilon}(\mathbf x,\mathbf y)\leq C_\varepsilon''w(B(\mathbf{x},\sqrt{t}))^{-1}  e^{-(1-2\varepsilon)\| \mathbf x-\mathbf y\|^2/4t}.
    \end{split}
\end{equation}
In the last inequality we have used Proposition~\ref{propo:heat}. Combining~\eqref{eq:basic},~\eqref{eq:I21}, and~\eqref{eq:I22} we obtain  
\begin{equation}
    \begin{split}
         \Big(2\mathbf N-2N+\frac{\| \mathbf x-\mathbf y\|^2}{t}\Big)h_t(\mathbf x,\mathbf y)&\leq (1-\varepsilon)\frac{\|\mathbf x-\mathbf y\|^2}{t}h_t(\mathbf x,\mathbf y)\\
         &+ 4C_\varepsilon''w(B(\mathbf{x},\sqrt{t}))^{-1}  e^{-(1-2\varepsilon)\| \mathbf x-\mathbf y\|^2/4t}\\
         &+2\sum_{\alpha\in R} k(\alpha)h_t(\mathbf{x},\sigma_\alpha(\mathbf y)),
    \end{split}
\end{equation}
which finally leads to~\eqref{eq:upper_1}, because, by our assumption, $2\mathbf N-2N>0$. 
% \begin{equation}
%      h_t(\mathbf x,\mathbf y)
%         \leq  Cw(B(\mathbf{x},\sqrt{t}))^{-1} e^{-c\| \mathbf x-\mathbf y\|^2/t}+C \Big(1+\frac{\| \mathbf x-\mathbf y\|^2}{t}\Big)^{-1} \sum_{\alpha\in R} k(\alpha)h_t(\sigma_\alpha(\mathbf x),\mathbf y).
% \end{equation}
\end{proof}

Observe that our basic upper and lower bounds (see~\eqref{eq:lower_1} and~\eqref{eq:upper_1}) are of the same type and they differ by the constants in the exponent of the first component. 

{From now on the constants $C_1,c_1$ from Proposition \ref{propo:basic} are fixed.} 

\begin{remark}
{\normalfont The estimate~\eqref{eq:upper_1} together with~\eqref{eq:heat_JFAA} imply the known bounds  
\begin{equation}\label{eq:heat_studia}
    h_t(\mathbf x,\mathbf y) \leq Cw(B(\mathbf x,\sqrt{t}))^{-1}\Big(1+\frac{\| \mathbf x-\mathbf y\|}{\sqrt{t}}\Big)^{-2} e^{-cd(\mathbf x,\mathbf y)^2/t}
\end{equation}
see~\cite[Theorem 3.1]{DzH1}. An alternative proof of  \eqref{eq:heat_studia} which uses a Poincar\'e inequality was announced by W. Hebisch. 
}\end{remark}

{
  \section{The case of the dihedral group - proof of Theorem \ref{th:dihedral}}
  
Let $D_m$ be a regular $m$-polygon in $\mathbb R^2$, $m\geq 3$, such that the related root system $R$ consists of $2m$ vectors
 \begin{align*}
      \alpha_j=\sqrt{2}\left(\sin\left(\frac{\pi j}{m}\right),\cos\left(\frac{\pi j}{m}\right)\right),  \ \ \ j \in \{0,1,\ldots,2m-1\},
  \end{align*}
and the reflection group $G$ acts either by the symmetries $\sigma_{\alpha_j}$, or by the rotations $\sigma_{\alpha_j}\circ \sigma_{\alpha_i}$, $0 \leq i,j \leq 2m-1$. Consequently, $\max_{\mathbf x,\mathbf y\in\mathbb R^2} n (\mathbf x,\mathbf y)=2$.
  \begin{proof}[Proof of Theorem \ref{th:dihedral}] \ 
  Fix $0<c_u<c_1$, where $c_1$ is a constant from Proposition~\ref{propo:basic}.
  Let us consider three cases depending on the value of $n(\mathbf{x},\mathbf{y})$.

  {\bf Case $n(\mathbf x,\mathbf y)=0$.} By the definition of $n(\mathbf{x},\mathbf{y})$ (see~\eqref{eq:n}), in this case $\| \mathbf x-\mathbf y\|=d(\mathbf x,\mathbf y)$. Hence Proposition~\ref{propo:heat} reads 
\begin{equation}\label{upl_0} C_{\rm lower}w(B(\mathbf{x},\sqrt{t}))^{-1} e^{-c_{\rm lower}d(\mathbf{x},\mathbf{y})^2/t}\leq h_t(\mathbf{x},\mathbf{y}) \leq C_{\rm upper} w(B(\mathbf{x},\sqrt{t}))^{-1}  e^{-c_{\rm upper}d(\mathbf{x},\mathbf{y})^2/t}, 
\end{equation}
which are the desired estimates, since $\Lambda_D(\mathbf{x},\mathbf{y},t)=1$ in this case. 

{\bf Case $n(\mathbf x,\mathbf y)=1$.} 
Then, by the definition of $n(\mathbf{x},\mathbf{y})$ (see~\eqref{eq:n}), there is $\alpha_0\in R$ such that $n (\mathbf{x},\sigma_{\alpha_0}(\mathbf y))=0$, that is,  $\| \mathbf{x}-\sigma_{\alpha_0}(\mathbf y)\|=d(\mathbf x,\mathbf y)$.
Using~\eqref{eq:lower_1}, we get 
\begin{equation}
\begin{split}\label{eq:lower-l1}
     h_t(\mathbf x,\mathbf y)&\geq C_1^{-1}\Big(1+\frac{\|\mathbf x-\mathbf y\|}{\sqrt{t}}\Big)^{-2}\sum_{\alpha\in R}h_t(\mathbf{x},\sigma_{\alpha}(\mathbf y)) \\
     &\geq C_1^{-1}\Big(1+\frac{\|\mathbf x-\mathbf y\|}{\sqrt{t}}\Big)^{-2}h_t(\mathbf{x},\sigma_{\alpha_0}(\mathbf y))\\& \geq  C_{1}^{-1}C_{\rm lower}w(B(\mathbf{x},\sqrt{t}))^{-1} e^{-c_{\rm lower} d(\mathbf x,\mathbf y)^2/t}\Big(1+\frac{\|\mathbf x-\mathbf y\|}{\sqrt{t}}\Big)^{-2}, 
\end{split}
\end{equation} 
where in the last inequality we have used~\eqref{upl_0}.

In order to prove the upper bound, we use \eqref{eq:upper_1}, Proposition~\ref{propo:heat}  together with the inequality  $d(\mathbf{x},\mathbf{y}) \leq \|\mathbf{x}-\mathbf{y}\|$ and obtain 
\begin{equation}\label{eq:upper-l1}
     h_t(\mathbf x,\mathbf y)\leq C_u w(B(\mathbf{x},\sqrt{t}))^{-1} e^{-c_{u}d(\mathbf x,\mathbf y)^2/t} \Big(1+\frac{\| \mathbf x-\mathbf y\|}{\sqrt{t}}\Big)^{-2}.
\end{equation}

{\bf Case of $ n(\mathbf x,\mathbf y)=2$.} In the proof of the upper and lower bounds we use the fact that, 
in this case, $n(\mathbf{x},\sigma_\alpha(\mathbf y))=1$ for all $\alpha \in R$.
 
 We start by proving the lower bound. Using  \eqref{eq:lower_1} we have 
\begin{equation}\begin{split}
  h_t(\mathbf x,\mathbf y)& \geq   C_1^{-1}\left(1+\frac{\|\mathbf{x}-\mathbf{y}\|}{\sqrt{t}}\right)^{-2} \sum_{\alpha \in R} h_{t}(\mathbf{x},\sigma_{\alpha}(\mathbf{y}))) \\
  &\geq  C_1^{-2}C_{\rm lower}  w(B(\mathbf{x},\sqrt{t}))^{-1}e^{-c_{\rm lower}\frac{d(\mathbf x,\mathbf y)^2}{t}} \left(1+\frac{\|\mathbf{x}-\mathbf{y}\|}{\sqrt{t}}\right)^{-2}
  \sum_{\alpha\in R}  \left(1+\frac{\|\mathbf{x}-\sigma_{\alpha}(\mathbf{y})\|}{\sqrt{t}}\right)^{-2} ,
  \end{split}
\end{equation}
where in the last inequality we have used~\eqref{eq:lower-l1}, since $n(\mathbf{x},\mathbf \sigma_{\alpha}(\mathbf y))=1$ for all $\alpha \in R$.

 In order to obtain upper bound, we apply \eqref{eq:upper_1} and then \eqref{eq:upper-l1}, and get 
\begin{equation}\begin{split}\label{eq:upper_l2}
    h_t(\mathbf{x},\mathbf{y}) &\leq C_1\Bigg(w(B(\mathbf{x},\sqrt{t}))^{-1}e^{-c_{1}\frac{\|\mathbf{x}-\mathbf{y}\|^2}{t}}+\left(1+\frac{\|\mathbf{x}-\mathbf{y}\|}{\sqrt{t}}\right)^{-2}\sum_{\alpha \in R}h_{t}(\mathbf{x},\sigma_{\alpha}(\mathbf{y}))\Bigg)\\
    &\leq C_1 w(B(\mathbf{x},\sqrt{t}))^{-1}e^{-c_{1}\frac{\|\mathbf{x}-\mathbf{y}\|^2}{t}}\\
   &\  + C_1C_{u}w(B(\mathbf{x},\sqrt{t}))^{-1}e^{-c_{1}d(\mathbf x,\mathbf y)^2/t}\left(1+\frac{\|\mathbf{x}-\mathbf{y}\|}{\sqrt{t}}\right)^{-2}\sum_{\alpha \in R} \Bigg(1+\frac{\| \mathbf{x}-\sigma_\alpha (\mathbf y)\|}{\sqrt{t}}\Bigg)^{-2}.\\
\end{split}
\end{equation}
Let now $\alpha_0\in R$ be such that $\|\mathbf{x}-\sigma_{\alpha_0}(\mathbf y)\|=\min_{\alpha \in R} \| \mathbf{x}-\sigma_{\alpha}(\mathbf y)\|$. Then  
\begin{align*}
    d(\mathbf x,\mathbf y)\leq \|\mathbf{x}-\sigma_{\alpha_0}(\mathbf y)\|\leq \|\mathbf x-\mathbf y\|    
\end{align*}
(see Lemma \ref{lem:alpha}). Thus, from \eqref{eq:upper_l2} we conclude that 
\begin{equation}\begin{split}
    h_t(\mathbf{x},\mathbf{y}) \leq 
   &C'_{u}w(B(\mathbf{x},\sqrt{t}))^{-1}e^{-c_ud(\mathbf x,\mathbf y)^2/t} \left(1+\frac{\|\mathbf{x}-\mathbf{y}\|}{\sqrt{t}}\right)^{-2}\Big(1+\frac{\|\mathbf{x}- \sigma_{\alpha_0} (\mathbf y)\|}{\sqrt{t}}\Big)^{-2}, \\
\end{split} 
\end{equation}
which implies the desired estimate \eqref{eq:main_claim_D}.

 \end{proof} }
  
  \section{Proof of Theorem \ref{teo:1} }
  
  \subsection{Proof of the lower bound \texorpdfstring{~\eqref{eq:main_lower}}{1}}
  
  The proposition below combined with Corollary~\ref{coro:ciag_skracajacy} imply~\eqref{eq:main_lower}.
  
  \begin{proposition}\label{propo:lower_one}
  Assume that $C_{\rm lower},c_{\rm lower}$ are the constants from Proposition~\ref{propo:heat} and $C_1$ is the constant from Proposition~\ref{propo:basic}. For all $\mathbf{x},\mathbf{y} \in \mathbb{R}^N$, $t>0$, and $\boldsymbol{\alpha} \in \mathcal{A}(\mathbf{x},\mathbf{y})$ we have
  \begin{equation}\label{eq:induction}
  \begin{split}
      h_{t}(\mathbf{x},\mathbf{y}) &\geq C_1^{-\ell(\boldsymbol{\alpha})} \rho_{\boldsymbol{\alpha}}(\mathbf{x},\mathbf{y},t)h_{t}(\mathbf{x},\sigma_{\boldsymbol{\alpha}}(\mathbf{y})) \\&\geq  C_{1}^{-\ell(\boldsymbol{\alpha})}C_{\rm lower} w(B(\mathbf{x},\sqrt{t}))^{-1}e^{-c_{\rm lower}\frac{d(\mathbf{x},\mathbf{y})^2}{t}}\rho_{\boldsymbol{\alpha}}(\mathbf{x},\mathbf{y},t) 
     \end{split}
  \end{equation}
  \end{proposition}

  \begin{proof}
  The proof is by induction with respect to $m=\ell(\boldsymbol{\alpha})$. For $m=0$ and $m=1$ the claim is a consequence of Proposition~\ref{propo:heat}  and~\eqref{eq:lower_1} (see also~\eqref{eq:lower-l1}). Assume that~\eqref{eq:induction} holds for all $\mathbf{x}_1,\mathbf{y}_1 \in \mathbb{R}^N$, $t_1>0$, and $\widetilde{\boldsymbol{\alpha}} \in \mathcal{A}(\mathbf{x}_1,\mathbf{y}_1)$ such that $\ell(\widetilde{\boldsymbol{\alpha}})=m$. Let $\boldsymbol{\alpha}=(\alpha_1,\alpha_2,\ldots,\alpha_{m+1}) \in \mathcal{A}(\mathbf{x},\mathbf{y})$ be such that $\ell(\boldsymbol{\alpha})=m+1$. By~\eqref{eq:lower_1} we have
  \begin{align*}
      h_{t}(\mathbf{x},\mathbf{y}) \geq C_{1}^{-1}\left(1+\frac{\|\mathbf{x}-\mathbf{y}\|}{\sqrt{t}}\right)^{-2}h_{t}(\mathbf{x},\sigma_{\alpha_1}(\mathbf{y})).
  \end{align*}
  Note that $\boldsymbol{\alpha} \in \mathcal{A}(\mathbf{x},\mathbf{y})$ implies that the sequence $\widetilde{\boldsymbol{\alpha}}=(\alpha_2,\ldots,\alpha_{m+1})$ belongs to $\mathcal{A}(\mathbf{x},\sigma_{\alpha_1}(\mathbf{y}))$ and, obviously, $\ell(\widetilde{\boldsymbol{\alpha}})=m$. Therefore, the claim is a consequence of the induction hypothesis applied to $\mathbf{x}$, $\sigma_{\alpha_1}(\mathbf{y})$, and $\widetilde{\boldsymbol{\alpha}}$, and the fact that, by the definition of $\rho_{\boldsymbol{\alpha}}(\mathbf{x},\mathbf{y},t)$ and $\rho_{\widetilde{\boldsymbol{\alpha}}}(\mathbf{x},\sigma_{\alpha_1}(\mathbf{y}),t)$ (see~\eqref{eq:rho}), we have
  \begin{align*}
      \rho_{\boldsymbol{\alpha}}(\mathbf{x},\mathbf{y},t)=\left(1+\frac{\|\mathbf{x}-\mathbf{y}\|}{\sqrt{t}}\right)^{-2}\rho_{\widetilde{\boldsymbol{\alpha}}}(\mathbf{x},\sigma_{\alpha_1}(\mathbf{y}),t).
  \end{align*}
  \end{proof}
  
  \subsection{Proof of the upper bound \texorpdfstring{~\eqref{eq:main_claim}}{1}}
  
  Let us begin with a corollary which follows by Proposition~\ref{propo:lower_one}.
  \begin{corollary}\label{coro:dolne_po_skracajacym}
  Assume that $c_{\rm lower}$ is the constant from Proposition~\ref{propo:heat}. Then there is a constant $C_2>0$ such that for all $\mathbf{x},\mathbf{y} \in \mathbb{R}^N$ and $t>0$ we have
  \begin{equation}\label{eq:lower2G}
      C_2^{-1}w(B(\mathbf{x},\sqrt{t}))^{-1}e^{-c_{\rm lower}\frac{d(\mathbf{x},\mathbf{y})^2}{t}}\left(1+\frac{\|\mathbf{x}-\mathbf{y}\|}{\sqrt{t}}\right)^{-2|G|} \leq h_t(\mathbf{x},\mathbf{y}).
  \end{equation}
  \end{corollary}
  
  \begin{proof}
If $n(\mathbf x,\mathbf y)=0$, then \eqref{eq:lower2G} holds by Proposition \ref{propo:heat}, because $d(\mathbf x,\mathbf y)=\| \mathbf x-\mathbf y\|$ in this case. For  fixed  $\mathbf{x},\mathbf{y} \in \mathbb{R}^N$ such that $n(\mathbf x,\mathbf y)\geq 1$, let $\boldsymbol{\alpha}=(\alpha_1,\alpha_2,\ldots,\alpha_m)$,   $m\leq |G|$, be as in Corollary~\ref{coro:ciag_skracajacy}. Then, thanks to ~\eqref{eq:ciag_skracajacy}, we have
   \begin{align*}
       \rho_{\boldsymbol{\alpha}}(\mathbf{x},\mathbf{y},t) \geq \left(1+\frac{\|\mathbf{x}-\mathbf{y}\|}{\sqrt{t}}\right)^{-2|G|},
   \end{align*}
   so the claim follows by Proposition~\ref{propo:lower_one}.
\end{proof}
From now on the constant $C_2$ from Corollary \ref{coro:dolne_po_skracajacym} is fixed.

\begin{proposition}\label{propo:d_small_t_small}
Let $C_1>0$ be the constant from Proposition~\ref{propo:basic}. There is a constant $c_2>\max(1,{4C_1|G|})$ such that for all $\mathbf{x},\mathbf{y} \in \mathbb{R}^N$ and $t>0$ satisfying
\begin{equation}\label{eq:assumptions}
    \|\mathbf{x}-\mathbf{y}\|>c_2d(\mathbf{x},\mathbf{y}) \text{ and } \|\mathbf{x}-\mathbf{y}\|>c_2\sqrt{t} 
\end{equation}
we have
\begin{equation}
    h_t(\mathbf{x},\mathbf{y}) \leq 2C_1 \left(1+\frac{\|\mathbf{x}-\mathbf{y}\|}{\sqrt{t}}\right)^{-2}\sum_{\alpha \in R}h_t(\mathbf{x},\sigma_{\alpha}(\mathbf{y})).
\end{equation}
\end{proposition}

\begin{remark}\normalfont
The condition "$c_2>\max(1,{4C_1|G|})$" occurs in the formulation of the proposition for some technical reasons and it will be used later on in the proof of~\eqref{eq:main_claim}.
\end{remark}

\begin{proof}
Thanks to~\eqref{eq:upper_1} and the fact that $0<h_t(\mathbf{x},\mathbf{y})<\infty$ it is enough to show
\begin{equation}\label{eq:goal_1}
    w(B(\mathbf{x},\sqrt{t}))^{-1}e^{-c_1\frac{\|\mathbf{x}-\mathbf{y}\|^2}{t}} \leq \frac{1}{2C_1} h_t(\mathbf{x},\mathbf{y}).
\end{equation}
To this end, by Corollary~\ref{coro:dolne_po_skracajacym}, we get
\begin{align*}
    h_t(\mathbf{x},\mathbf{y}) \geq C_2^{-1}w(B(\mathbf{x},\sqrt{t}))^{-1}e^{-c_{\rm lower}\frac{d(\mathbf{x},\mathbf{y})^2}{t}} \left(1+\frac{\|\mathbf{x}-\mathbf{y}\|}{\sqrt{t}}\right)^{-2|G|},
\end{align*}
so~\eqref{eq:goal_1} is a consequence of the fact that taking $c_2>0$ large enough in~\eqref{eq:assumptions}, we have
\begin{align*}
    C_2^{-1} \left(1+\frac{\|\mathbf{x}-\mathbf{y}\|}{\sqrt{t}}\right)^{-2|G|}e^{-c_{\rm lower}\frac{d(\mathbf{x},\mathbf{y})^2}{t}} \geq 2C_1 e^{-c_1\frac{\|\mathbf{x}-\mathbf{y}\|^2}{t}}.
\end{align*}
\end{proof}

From this moment the constant $c_2$ from Proposition \ref{propo:d_small_t_small} is fixed.
 
\begin{proposition}\label{propo:small_case} Assume that $c_{\rm upper}$ is the constant from Proposition~\ref{propo:heat} and $c_2$ is the same as in Proposition~\ref{propo:d_small_t_small}.
Let $c_3<c_{\rm upper}$. There is a constant $C_3>0$ such that for all $\mathbf{x},\mathbf{y} \in \mathbb{R}^N$ and $t>0$ such that
\begin{align*}
\|\mathbf{x}-\mathbf{y}\| \leq c_2\sqrt{t} \text{ or }\|\mathbf{x}-\mathbf{y}\| \leq c_2d(\mathbf{x},\mathbf{y})    
\end{align*}
there is $\boldsymbol{\alpha}\in \mathcal{A}(\mathbf{x},\mathbf{y})$, $\ell(\boldsymbol{\alpha}) \leq |G|$, such that
\begin{equation}\label{eq:upper_rho_a}
\begin{split}
    h_t(\mathbf{x},\mathbf{y}) &\leq C_3w(B(\mathbf{x},\sqrt{t}))^{-1}e^{-c_3\frac{d(\mathbf{x},\mathbf{y})^2}{t}} \rho_{\boldsymbol{\alpha}}(\mathbf{x},\mathbf{y},t).
\end{split}
\end{equation}
\end{proposition}

\begin{proof}
If $n(\mathbf{x},\mathbf{y})=0$, then one can take $\boldsymbol{\alpha}=\emptyset$ and the claim is a consequence of Proposition~\ref{propo:heat}. Assume that $n(\mathbf{x},\mathbf{y})>0$. For fixed $\mathbf{x},\mathbf{y}$, let   and $\boldsymbol{\alpha}=(\alpha_1,\alpha_2,\ldots,\alpha_m)$, $m \leq |G|$, be as in Corollary~\ref{coro:ciag_skracajacy}. If $\|\mathbf{x}-\mathbf{y}\| \leq c_2\sqrt{t}$, then the claim is satisfied by Proposition~\ref{propo:heat} and~\eqref{eq:ciag_skracajacy}, and we may take $c_3=c_{\rm upper}$ in the inequality~\eqref{eq:upper_rho_a}. If $\|\mathbf{x}-\mathbf{y}\| \leq c_2d(\mathbf{x},\mathbf{y})$, then by Proposition~\ref{propo:heat} we get
\begin{align*}
    h_t(\mathbf{x},\mathbf{y}) \leq C_{\rm upper}w(B(\mathbf{x},\sqrt{t}))^{-1}e^{c_{\rm upper}\frac{d(\mathbf{x},\mathbf{y})^2}{t}}=C_{\rm upper}w(B(\mathbf{x},\sqrt{t}))^{-1}e^{-c_{3}\frac{d(\mathbf{x},\mathbf{y})^2}{t}}e^{-(c_{\rm upper}-c_3)\frac{d(\mathbf{x},\mathbf{y})^2}{t}}. 
\end{align*}
Moreover, the assumption $\|\mathbf{x}-\mathbf{y}\| \leq c_2d(\mathbf{x},\mathbf{y})$ implies
\begin{align*}
    e^{-(c_{\rm upper}-c_3)\frac{d(\mathbf{x},\mathbf{y})^2}{t}} \leq e^{-(c_{\rm upper}-c_3)\frac{\|\mathbf{x}-\mathbf{y}\|^2}{c_2^2t}},
\end{align*}
so the claim follows by the fact that there is $C>0$ such that
\begin{equation*}
      e^{-(c_{\rm upper}-c_3)\frac{\|\mathbf{x}-\mathbf{y}\|^2}{c_2^2t}} \leq C\left(1+\frac{\|\mathbf{x}-\mathbf{y}\|}{\sqrt{t}}\right)^{-2|G|} \leq C\rho_{\boldsymbol{\alpha}}(\mathbf{x},\mathbf{y},t),
\end{equation*}
where the second inequality is a consequence of~\eqref{eq:ciag_skracajacy}.
\end{proof}

From now on the constant $C_3, c_3$ from Proposition \ref{propo:small_case} are fixed.

\begin{proof}[Proof of~\eqref{eq:main_claim}]
Let $c_2$ be the constant from Proposition~\ref{propo:d_small_t_small}. Fix $\mathbf{x},\mathbf{y} \in \mathbb{R}^N$ and $t>0$ and consider
\begin{align*}
    G_0:=\{\sigma \in G\;:\;\text{ assumption}~\eqref{eq:assumptions} \text{ is not satisfied for }\mathbf{x},\sigma(\mathbf{y}),t\}.
\end{align*}
Note that $G_0 \neq \emptyset$, because there is $\sigma_0 \in G$ such that $\|\mathbf{x}-\sigma_0(\mathbf{y})\|=d(\mathbf{x},\mathbf{y})=d(\mathbf{x},\sigma_0(\mathbf{y}))$, so the assumption~\eqref{eq:assumptions} is not satisfied for $\mathbf{x}$, $\sigma_0(\mathbf{y})$, $t$. We will prove~\eqref{eq:main_claim} for $h_t(\mathbf{x},\sigma(\mathbf{y}))$ for all $\sigma \in G$. Note that, by the definition of $G_0$, if $\sigma \in G_0$, then by Proposition~\ref{propo:small_case} we have
\begin{equation}\label{eq:main_claim2_g0}
    h_t(\mathbf{x}, \sigma(\mathbf{y})) \leq C_3 w(B(\mathbf{x},\sqrt{t}))^{-1}e^{-c_3\frac{d(\mathbf{x},\mathbf{y})^2}{t}} \sum_{\boldsymbol{\alpha} \in \mathcal{A}(\mathbf{x},\sigma(\mathbf{y}))\;:\;\ell(\boldsymbol{\alpha}) \leq |G|}\rho_{\boldsymbol{\alpha}}(\mathbf{x},\sigma(\mathbf{y}),t).
\end{equation}
If $G=G_0$, the proof is complete. Assume that $G_0\ne G$.  Consider all the values  $h_t(\mathbf{x},\sigma(\mathbf{y}))$ for all $\sigma \not\in G_0$ and list them in a decreasing sequence, that is,   $G\setminus G_0=\{\sigma_1,\sigma_2,\ldots,\sigma_m\}$ and 
\begin{equation}\label{eq:order}
    h_t(\mathbf{x},\sigma_1(\mathbf{y})) \geq h_t(\mathbf{x},\sigma_2(\mathbf{y})) \geq \ldots \geq h_t(\mathbf{x},\sigma_{m}(\mathbf{y})).
\end{equation}
For $1 \leq j \leq m$ let us denote
\begin{equation}
    G_j:=G_0 \cup \{\sigma_1,\ldots,\sigma_j\}. 
\end{equation}
We will prove by induction on $j$ that for all $1 \leq j \leq m$ we have
\begin{equation}\label{eq:main_claim2}
    h_t(\mathbf{x}, \sigma_j(\mathbf{y})) \leq C_3(2C_1|G|)^{j}w(B(\mathbf{x},\sqrt{t}))^{-1}e^{-c_3\frac{d(\mathbf{x},\mathbf{y})^2}{t}} \sum_{\boldsymbol{\alpha} \in \mathcal{A}(\mathbf{x},\sigma_j(\mathbf{y}))\;:\;\ell(\boldsymbol{\alpha}) \leq |G|+j}\rho_{\boldsymbol{\alpha}}(\mathbf{x},\sigma_j(\mathbf{y}),t),
\end{equation}
where $C_3,c_3$ are the constants from Proposition~\ref{propo:small_case}, where $C_1$ is the constant from Proposition \ref{propo:basic}. We have already remarked that \eqref{eq:main_claim2} is satisfied for $\sigma\in G_0$ with $j=0$, (see~\eqref{eq:main_claim2_g0}). Suppose that the estimate~\eqref{eq:main_claim2} holds for all $h_t(\mathbf{x},\sigma(\mathbf{y}))$ where $\sigma \in G_j$. We will prove~\eqref{eq:main_claim2} for $h_t(\mathbf{x},\sigma_{j+1}(\mathbf{y}))$. By the fact that $\sigma_{j+1} \not \in G_0$ and Proposition~\ref{propo:d_small_t_small} we get
\begin{equation}\label{eq:C2_applied}
    h_t(\mathbf{x},\sigma_{j+1}(\mathbf{y})) \leq 2C_1 \left(1+\frac{\|\mathbf{x}-\sigma_{j+1}(\mathbf{y})\|}{\sqrt{t}}\right)^{-2}\sum_{\alpha \in R}h_t(\mathbf{x},\sigma_{\alpha} \circ \sigma_{j+1}(\mathbf{y})).
\end{equation}
Recall that $\sigma_{j+1} \not\in G_0$, so the assumption~\eqref{eq:assumptions} is satisfied for $\mathbf{x},\sigma_{j+1}(\mathbf{y})$, and $t$. Therefore, since $c_2>\max(1,4C_1|G|)$ in Proposition~\ref{propo:d_small_t_small}, we get
\begin{equation}\label{eq:c2_is_large}
    2C_1 \left(1+\frac{\|\mathbf{x}-\sigma_{j+1}(\mathbf{y})\|}{\sqrt{t}}\right)^{-2} <\frac{1}{8C_1|G|^2}< \frac{1}{2|G|}.
\end{equation}
Note that there is $\alpha \in R$ such that $\sigma_{\alpha} \circ \sigma_{j+1} \in G_j$, because otherwise~\eqref{eq:order},~\eqref{eq:C2_applied}, and~\eqref{eq:c2_is_large} lead us to a contradiction. Let $\alpha_0 \in R$ be such that $\sigma_{\alpha_0} \circ \sigma_{j+1} \in G_j$ and $h_t(\mathbf{x},\sigma_{\alpha_0}\circ \sigma_{j+1}(\mathbf{y}))$ is the largest possible, that is,
\begin{align*}
    h_t(\mathbf{x},\sigma_{\alpha_0}\circ \sigma_{j+1}(\mathbf{y}))=\max_{\alpha\in R_+, \; \sigma_{\alpha} \circ \sigma_{j+1} \in G_j} h_t(\mathbf{x},\sigma_{\alpha}\circ \sigma_{j+1}(\mathbf{y})).
\end{align*}

Then by~\eqref{eq:C2_applied} we get
\begin{align*}
    h_t(\mathbf{x},\sigma_{j+1}(\mathbf{y})) \leq 2C_1|G| \left(1+\frac{\|\mathbf{x}-\sigma_{j+1}(\mathbf{y})\|}{\sqrt{t}}\right)^{-2}h_{t}(\mathbf{x},\sigma_{\alpha_0}\circ \sigma_{j+1}(\mathbf{y})).
\end{align*}
By the induction hypothesis ~\eqref{eq:main_claim2} already holds for $h_{t}(\mathbf{x},\sigma_{\alpha_0}\circ \sigma_{j+1}(\mathbf{y}))$. Hence
\begin{align*}
    h_t(\mathbf{x},\sigma_{j+1}(\mathbf{y})) &\leq C_3(2C_1|G|)^{j+1}w(B(\mathbf{x},\sqrt{t}))^{-1}e^{-c_3\frac{d(\mathbf{x},\mathbf{y})^2}{t}}\\&\left(1+\frac{\|\mathbf{x}-\sigma_{j+1}(\mathbf{y})\|}{\sqrt{t}}\right)^{-2} \sum_{\boldsymbol{\alpha} \in \mathcal{A}(\mathbf{x},\sigma_{\alpha_0} \circ \sigma_{j+1}(\mathbf{y}))\;:\;\ell(\boldsymbol{\alpha}) \leq |G|+j}\rho_{\boldsymbol{\alpha}}(\mathbf{x},\sigma_{\alpha_0} \circ \sigma_{j+1}(\mathbf{y}),t)
\end{align*}
For any sequence
\begin{align*}
    \boldsymbol{\alpha}=(\alpha_1,\alpha_2,\ldots,\alpha_m)
\end{align*}
from $\mathcal{A}(\mathbf{x},\sigma_{\alpha_0}\circ \sigma_{j+1}(\mathbf{y}))$ with $\ell(\boldsymbol \alpha)\leq |G|+j$, 
we define the new sequence 
 $$ \widetilde{\boldsymbol{\alpha}}=(\alpha_0,\alpha_1,\alpha_2,\ldots,\alpha_m)
$$
from $\mathcal{A}(\mathbf{x},\sigma_{j+1}(\mathbf{y}))$ 
 satisfying  $ \ell(\widetilde{\boldsymbol{\alpha}})\leq |G|+j+1$.  Moreover, 
 \begin{align*}
 \left(1+\frac{\|\mathbf{x}-\sigma_{j+1}(\mathbf{y})\|}{\sqrt{t}}\right)^{-2} \rho_{\boldsymbol{\alpha}}(\mathbf{x},\sigma_{\alpha_0} \circ \sigma_{j+1}(\mathbf{y}),t)= \rho_{\widetilde{\boldsymbol \alpha} } (\mathbf x,\sigma_{j+1} (\mathbf y),t),
\end{align*}
so~\eqref{eq:main_claim2} for $h_t(\mathbf{x},\sigma_{j+1}(\mathbf{y}))$ is proved.
\end{proof}

\section{Applications of Theorem \texorpdfstring{\ref{teo:1}}{1.1}}

\subsection{Regularity of the heat kernel}
The following theorem can be consider as an improvement of the estimates~\cite[Theorem 4.1 (b)]{ADzH}.

\begin{theorem}\label{teo:holder_est}
Let $m$ be a non-negative integer. There are constants $C_4,c_4>0$ such that for all $\mathbf{x},\mathbf{y},\mathbf{y}' \in \mathbb{R}^N$ and $t>0$ satisfying $\|\mathbf{y}-\mathbf{y}'\|<\frac{\sqrt{t}}{2}$ we have
\begin{equation}\label{eq:h_t_Holder_d}
    |\partial_t^m h_t(\mathbf{x},\mathbf{y})-\partial_t^m h_{t}(\mathbf{x},\mathbf{y}')| \leq C_4t^{-m}\frac{\|\mathbf{y}-\mathbf{y}'\|}{\sqrt{t}}h_{c_4t}(\mathbf{x},\mathbf{y}).
\end{equation}
The constant $c_4$ does not depend on $m$.
\end{theorem}

In the proof of Theorem~\ref{teo:holder_est} we will need the following lemma.

\begin{lemma}\label{lem:small_def}
  There are constants $C_5,c_5>0$ such that for all $\mathbf{x},\mathbf{y},\mathbf{y}' \in \mathbb{R}^N$ and $t>0$ satisfying $\|\mathbf{y}-\mathbf{y}'\|<\frac{\sqrt{t}}{2}$ we have
  \begin{equation}\label{eq:small_def}
      h_{t}(\mathbf{x},\mathbf{y}) \leq C_5 h_{c_5 t}(\mathbf{x},\mathbf{y}').
  \end{equation}
\end{lemma}

\begin{proof}
Fix $\mathbf{x},\mathbf{y},\mathbf{y}' \in \mathbb{R}^{N}$ and $t>0$ such that $\|\mathbf{y}-\mathbf{y}'\| < \frac{\sqrt{t}}{2}$. By Theorem~\ref{teo:1} we have
\begin{equation}\label{eq:upper_app}
    h_t(\mathbf{x},\mathbf{y}) \leq C_{u}w(B(\mathbf{x},\sqrt{t}))^{-1}e^{-c_{ u}\frac{d(\mathbf{x},\mathbf{y})^2}{t}} \Lambda(\mathbf{x},\mathbf{y},t),
\end{equation}
where
\begin{equation}\label{eq:Lambda_app}
    \Lambda (\mathbf x,\mathbf y,t)=\sum_{\boldsymbol \alpha \in \mathcal{A}(\mathbf{x},\mathbf{y}), \;\ell (\boldsymbol \alpha) \leq 2|G|} \rho_{\boldsymbol \alpha}(\mathbf x,\mathbf y,t).
\end{equation}
Let us consider a sequence $\boldsymbol{\alpha} \in \mathcal{A}(\mathbf{x},\mathbf{y})$ such that $\ell(\boldsymbol{\alpha}) \leq 2|G|$. We shall prove that 
\begin{equation}\label{eq:replace_y_by_y'}
    \rho_{\boldsymbol{\alpha}}(\mathbf{x},\mathbf{y},t) \leq 2^{2\ell(\boldsymbol{\alpha})}\rho_{\boldsymbol{\alpha}}(\mathbf{x},\mathbf{y}',t).
\end{equation}
If $\ell(\boldsymbol{\alpha})=0$, then \eqref{eq:replace_y_by_y'} is trivial. 
 If $\boldsymbol{\alpha}=(\alpha_1,\alpha_2,\ldots,\alpha_m) \in \mathcal{A}(\mathbf{x},\mathbf{y})$ and  $\|\mathbf{y}-\mathbf{y}'\|<\frac{\sqrt{t}}{2}$ , then for any $1 \leq j \leq \ell(\boldsymbol{\alpha})$, we have
\begin{align*}
    &\frac{\|\mathbf{x}-\sigma_{\alpha_j}\circ \sigma_{\alpha_{j-1}} \circ \ldots \circ \sigma_{\alpha_1}(\mathbf{y}')\|}{\sqrt{t}}+1 \\&\leq \frac{\|\mathbf{x}-\sigma_{\alpha_j}\circ \sigma_{\alpha_{j-1}} \circ \ldots \circ \sigma_{\alpha_1}(\mathbf{y})\|+\|\sigma_{\alpha_j}\circ \sigma_{\alpha_{j-1}} \circ \ldots \circ \sigma_{\alpha_1}(\mathbf{y})-\sigma_{\alpha_j}\circ \sigma_{\alpha_{j-1}} \circ \ldots \circ \sigma_{\alpha_1}(\mathbf{y}')\|}{\sqrt{t}}+1\\&=\frac{\|\mathbf{x}-\sigma_{\alpha_j}\circ \sigma_{\alpha_{j-1}} \circ \ldots \circ \sigma_{\alpha_1}(\mathbf{y})\|+\|\mathbf{y}-\mathbf{y}'\|}{\sqrt{t}}+1 \leq 2 \left(\frac{\|\mathbf{x}-\sigma_{\alpha_j}\circ \sigma_{\alpha_{j-1}} \circ \ldots \circ \sigma_{\alpha_1}(\mathbf{y})\|}{\sqrt{t}}+1\right).
\end{align*}
Hence, by the definition of $\rho_{\boldsymbol{\alpha}}(\mathbf{x},\mathbf{y},t)$ (see~\eqref{eq:rho}), we obtain \eqref{eq:replace_y_by_y'}.

Now, for $\boldsymbol{\alpha}\in \mathcal A(\mathbf x,\mathbf y)$, $ \ell(\boldsymbol{\alpha})\leq 2|G|$,  we are going to define a new sequence  $\widetilde{\boldsymbol{\alpha}}$ of elements of $R$ such that $\widetilde{\boldsymbol{\alpha}} \in \mathcal{A}(\mathbf{x},\mathbf{y}')$, $\ell(\widetilde{\boldsymbol{\alpha}}) \leq 4|G|$, and
\begin{equation}\label{eq:new_rho}
    \rho_{\boldsymbol{\alpha}}(\mathbf{x},\mathbf{y}',t) \leq 2^{2|G|}\left(1+\frac{d(\mathbf{x},\mathbf{y})}{\sqrt{t}}\right)^{2|G|} \rho_{\widetilde{\boldsymbol{\alpha}}}(\mathbf{x},\mathbf{y}',t).
\end{equation}
To this end, let us consider two cases.

\textbf{Case 1.} $\boldsymbol{\alpha} \in \mathcal{A}(\mathbf{x},\mathbf{y}')$. Then we set $\widetilde{\boldsymbol{\alpha}}:=\boldsymbol{\alpha}$. Clearly, in this case~\eqref{eq:new_rho} is satisfied.

\textbf{Case 2.} $\boldsymbol{\alpha} \not\in \mathcal{A}(\mathbf{x},\mathbf{y}')$. Let $\boldsymbol{\beta}=(\beta_1,\beta_2,\ldots,\beta_{m_1})$ be a sequence from Corollary~\ref{coro:ciag_skracajacy} chosen for the points $\mathbf{x}$ and $\sigma_{\boldsymbol{\alpha}}(\mathbf{y}')$. In particular, by~\eqref{eq:ciag_skracajacy}, then by the fact that $\boldsymbol{\alpha} \in \mathcal{A}(\mathbf{x},\mathbf{y})$ and $\|\mathbf{y}-\mathbf{y}'\| \leq \frac{\sqrt{t}}{2}$, for any $1 \leq j \leq \ell(\boldsymbol{\beta})$, we have
\begin{equation*}\begin{split}
    \|\mathbf{x}-\sigma_{\beta_j} \circ \sigma_{\beta_{j-1}} \circ \ldots \circ \sigma_{\beta_1} \circ \sigma_{\boldsymbol{\alpha}}(\mathbf{y}')\| &\leq \|\mathbf{x}-\sigma_{\boldsymbol{\alpha}}(\mathbf{y}')\| \\&\leq \|\mathbf{x}-\sigma_{\boldsymbol{\alpha}}(\mathbf{y})\|+\|\sigma_{\boldsymbol{\alpha}}(\mathbf{y})-\sigma_{\boldsymbol{\alpha}}(\mathbf{y}')\|\\&=\|\mathbf{x}-\sigma_{\boldsymbol{\alpha}}(\mathbf{y})\|+\|\mathbf{y}-\mathbf{y}'\|\\& \leq  d(\mathbf{x},\mathbf{y})+\frac{\sqrt{t}}{2}.
    \end{split}
\end{equation*}
Consequently,
\begin{equation}\label{eq:two_seq_compare}
    1 \geq \rho_{\boldsymbol{\beta}}(\mathbf{x},\sigma_{\boldsymbol{\alpha}}(\mathbf{y}'),t) \geq 2^{-2|G|} \Big(1+\frac{d(\mathbf x,\mathbf y)}{\sqrt{t} } \Big)^{-2|G|}.
\end{equation}
We set 
\begin{align*}
 \widetilde{\boldsymbol{\alpha}}:=(\beta_1,\beta_2,\ldots,\beta_{m_1})  \quad &\text{if } \ell(\boldsymbol{\alpha})=0,\\
  \widetilde{\boldsymbol{\alpha}}:=(\alpha_1,\alpha_2,\ldots,\alpha_m,\beta_1,\beta_2,\ldots,\beta_{m_1}) \quad  &\text{if } {\boldsymbol{\alpha}}=(\alpha_1,\alpha_2,\ldots,\alpha_m).
\end{align*} 
Then, by the definition of $\boldsymbol{\beta}$, we have $\widetilde{\boldsymbol{\alpha}} \in \mathcal{A}(\mathbf{x},\mathbf{y}')$, $\ell(\widetilde{\boldsymbol{\alpha}}) \leq 4|G|$. Moreover, by the definition of $\rho_{\boldsymbol{\alpha}}(\mathbf{x},\mathbf{y}',t)$ and $\rho_{\widetilde{\boldsymbol{\alpha}}}(\mathbf{x},\mathbf{y}',t)$, and~\eqref{eq:two_seq_compare} we have
\begin{equation*}
    \rho_{\widetilde{\boldsymbol{\alpha}}}(\mathbf{x},\mathbf{y}',t)=\rho_{\boldsymbol{\alpha}}(\mathbf{x},\mathbf{y}',t)\rho_{\boldsymbol{\beta}}(\mathbf{x},\sigma_{\boldsymbol{\alpha}}(\mathbf{y}'),t) \geq 2^{-2|G|} \Big(1+\frac{d(\mathbf x,\mathbf y)}{\sqrt{t} }\Big)^{-2|G|}\rho_{\boldsymbol{\alpha}}(\mathbf{x},\mathbf{y}',t),
\end{equation*}
which implies~\eqref{eq:new_rho}. 

Applying~\eqref{eq:replace_y_by_y'} and~\eqref{eq:new_rho}, we have 
\begin{align*}
    \Lambda(\mathbf{x},\mathbf{y},t)=\sum_{\boldsymbol \alpha \in \mathcal{A}(\mathbf{x},\mathbf{y}), \;\ell (\boldsymbol \alpha) \leq 2|G|} \rho_{\boldsymbol \alpha}(\mathbf x,\mathbf y,t) \leq 2^{6|G|}\Big(1+\frac{d(\mathbf x,\mathbf y)}{\sqrt{t} }\Big)^{2|G|}\sum_{\boldsymbol \alpha \in \mathcal{A}(\mathbf{x},\mathbf{y}'), \;\ell (\boldsymbol \alpha) \leq 4|G|} \rho_{\boldsymbol \alpha}(\mathbf x,\mathbf y',t),
\end{align*}
which, together with~\eqref{eq:upper_app}, gives 
\begin{equation}\label{eq:holder_lemma_final}
\begin{split}
    h_t(\mathbf{x},\mathbf{y}) &\leq 2^{6|G|}C_u \left(1+\frac{d(\mathbf{x},\mathbf{y})}{\sqrt{t}}\right)^{2|G|}w(B(\mathbf{x},\sqrt{t}))^{-1}e^{-c_u\frac{d(\mathbf{x},\mathbf{y})^2}{t}}\sum_{\boldsymbol \alpha \in \mathcal{A}(\mathbf{x},\mathbf{y}'), \;\ell (\boldsymbol \alpha) \leq 4|G|} \rho_{\boldsymbol \alpha}(\mathbf x,\mathbf y',t)\\&\leq C_u' w(B(\mathbf{x},\sqrt{t}))^{-1}e^{-c_u'\frac{d(\mathbf{x},\mathbf{y})^2}{t}}\sum_{\boldsymbol \alpha \in \mathcal{A}(\mathbf{x},\mathbf{y}'), \;\ell (\boldsymbol \alpha) \leq 4|G|} \rho_{\boldsymbol \alpha}(\mathbf x,\mathbf y',t).
    \end{split}
\end{equation}
Note that $\|\mathbf{y}-\mathbf{y}'\| <\frac{\sqrt{t}}{2}$ implies
\begin{align*}
    2d(\mathbf{x},\mathbf{y})^2 \geq d(\mathbf{x},\mathbf{y}')^2-2\|\mathbf{y}-\mathbf{y}'\|^2 \geq d(\mathbf{x},\mathbf{y}')^2-\frac{t}{2}.
\end{align*}
Therefore, by~\eqref{eq:holder_lemma_final} we get
\begin{equation}\label{eq:ff}
    h_t(\mathbf{x},\mathbf{y}) \leq C_u'' w(B(\mathbf{x},\sqrt{t}))^{-1}e^{-c_u'\frac{d(\mathbf{x},\mathbf{y}')^2}{2t}}\sum_{\boldsymbol \alpha \in \mathcal{A}(\mathbf{x},\mathbf{y}'), \;\ell (\boldsymbol \alpha) \leq 4|G|} \rho_{\boldsymbol \alpha}(\mathbf x,\mathbf y',t).
\end{equation}
Finally,~\eqref{eq:small_def} is a consequence of Proposition~\ref{propo:lower_one},~\eqref{eq:ff}, and~\eqref{eq:Lambda_2t}.
\end{proof}

\begin{proof}[Proof of Theorem \texorpdfstring{\ref{teo:holder_est}}{??}]
For $x \in \mathbb{R}$ and $t>0$ let us denote
\begin{align*}
    \widetilde{h_t}(x):={\boldsymbol c}_k^{-1}2^{-\mathbf{N}/2}t^{-\mathbf{N}/2}\exp\left(-\frac{x^2}{4t}\right).
\end{align*}
Then the formula~\eqref{heat:Rosler} reads
\begin{equation}\label{eq:modified_Roesler}
    h_t(\mathbf{x},\mathbf{y})=\int_{\mathbb{R}^N}\widetilde{h_t}(A(\mathbf{x},\mathbf{y},\eta))\, d\mu_{\mathbf{x}}(\eta).
\end{equation}
Observe that \hspace{.5mm}$\tilde{\mathfrak{h}}_t(x)\hspace{-.5mm}:=\hspace{-.25mm}\partial_x\hspace{.25mm}\partial_t^m\widetilde{h_t}(x)$ is equal to \hspace{.5mm}$\frac x{t^{m+1}} \tilde h_t(x)$ times a polynomial in $\frac{x^2}t$. Hence, for any non-negative integer $m$ there is a constant $C_m>0$ such that for all $x \in \mathbb{R}$ and $t>0$ we have
\begin{equation}\label{eq:C_m_der}
\bigl|\hspace{.25mm}{\tilde{\mathfrak{h}}_t}(x)\bigr|\leq C_m\,t^{-m-1\slash 2}\,\widetilde{h_{2t}}(x)\,.
\end{equation}
Further, by~\eqref{eq:modified_Roesler} and~\eqref{eq:C_m_der} we get
\begin{equation}\label{eq:comp}\begin{split}
|{\partial_t^mh_t}(\mathbf{x},\mathbf{y})-{\partial_t^mh_t}(\mathbf{x},\mathbf{y}')|
&=\Bigl|\int_{\mathbb{R}^N}\bigl\{{\partial_t^m\widetilde{h_t}}(A(\mathbf{x},\mathbf{y},\eta))-{\partial_t^m\widetilde{h_t}}(A(\mathbf{x},\mathbf{y}',\eta))\bigr\}d\mu_{\mathbf{x}}(\eta)\Bigr|\\
& =\Bigl|\int_{\mathbb{R}^N}\int_0^1{\frac{\partial}{\partial s}\partial_t^m\widetilde{h_t}}(A(\mathbf{x},{\underbrace{\mathbf{y}'\hspace{-1mm}+\hspace{-.5mm}s(\mathbf{y}\!-\hspace{-.5mm}\mathbf{y}')}_{\mathbf{y}_{\hspace{-.25mm}s}}},\eta))\,ds\,d\mu_{\mathbf{x}}(\eta)\Bigr|\\
&\le{\|}\mathbf{y}\!-\hspace{-.5mm}\mathbf{y}'{\|}\int_0^1\!
 \int_{\mathbb{R}^N}\,\bigl|\hspace{.25mm}{\tilde{\mathfrak{h}}_t}(A(\mathbf{x},\mathbf{y}_{\hspace{-.25mm}s},\eta))
 \bigr| \,d\mu_{\mathbf{x}}(\eta) \,ds\\
&\le{C_m\,t^{-m}}\,\frac{{\|}\mathbf{y}\!-\hspace{-.5mm}\mathbf{y}'{\|}}{\sqrt{t\,}}
 \int_0^1h_{2t}(\mathbf{x},\mathbf{y}_{\hspace{-.25mm}s})\,ds.
\end{split}\end{equation}
Finally, note that for any $s \in [0,1]$, we have
\begin{align*}
    \|\mathbf{y}-\mathbf{y}_s\| \leq \|\mathbf{y}-\mathbf{y}'\| <\frac{\sqrt{t}}{2},
\end{align*}
so, by Lemma~\ref{lem:small_def}, we get
\begin{align}\label{eq:y_and_y_s}
    h_{2t}(\mathbf{x},\mathbf{y}_{\hspace{-.25mm}s}) \leq C_5 h_{2c_5t}(\mathbf{x},\mathbf{y}).
\end{align}
Now~\eqref{eq:comp} together with~\eqref{eq:y_and_y_s} imply the desired estimate~\eqref{eq:h_t_Holder_d}. 
\end{proof}

\begin{remark} \normalfont
In the proof of Theorem~\ref{teo:holder_est}, we partially repeat the argument from~\cite[Theorem 4.1 (b)]{ADzH}. The novelty of the approach is using Lemma~\ref{lem:small_def} instead of Proposition~\ref{propo:heat} in estimating the last integral of~\eqref{eq:comp}.
\end{remark}

\subsection{Remark on a theorem of Gallardo and Rejeb} {
In \cite{GallardoRejeb}, the authors proved that the points $\sigma(\mathbf x)$, $\sigma \in G$, belong to the support of the measure $\mu_{\mathbf x}$ (see~\cite[Theorem A 3)]{GallardoRejeb}). 
% The following theorem, which is an application of Theorem \ref{teo:1},  gives more precise behavior of the measure $\mu_{\mathbf x}$ near these points.  
{Below, as an application of the estimates~\eqref{eq:main_lower} and~\eqref{eq:main_claim}, we provide another proof of this theorem. The proof, at the same time, gives a more precise behavior of the measure  $\mu_{\mathbf x}$ around these points.}

For $\mathbf{y} \in \mathbb{R}^N$ and $t>0$ we set
\begin{equation}\label{eq:U}
    U(\mathbf y,t):=\{ \eta\in \text{\rm conv}\, \mathcal O(\mathbf y): \| \mathbf y\|^2-\langle\mathbf y,\eta\rangle \leq t  \},
\end{equation}
\begin{equation}\label{eq:V}
    V(\mathbf y,t):=(\text{\rm conv}\, \mathcal O(\mathbf y)) \setminus U(\mathbf{y},t)=\{ \eta\in \text{\rm conv}\, \mathcal O(\mathbf y): \| \mathbf y\|^2-\langle \mathbf y, \eta \rangle > t  \}.
\end{equation}

\begin{theorem}\label{teo:rejeb}  
There is a constant $C_6>0$ such that for all $\mathbf{x} \in \mathbb{R}^N$, $t>0$, and $\sigma \in G$ we have 
\begin{equation}\label{eq:precise_measure}
    C^{-1}_6\frac{t^{\mathbf N/2}\Lambda(\mathbf x,\sigma(\mathbf x),t)}{w(B(\mathbf x,\sqrt{t}))}\leq  \mu_{\mathbf x}(U(\sigma(\mathbf x),t))\leq C_6\frac{t^{\mathbf N/2}\Lambda(\mathbf x,\sigma(\mathbf x),t)}{w(B(\mathbf x,\sqrt{t}))}.
\end{equation}

\end{theorem}

\begin{proof}
Let $\mathbf y=\sigma(\mathbf x)$. Then $d(\mathbf x,\mathbf y)=0$. Moreover, by the definition of $A(\mathbf{x},\mathbf{y},\eta)$ (see~\eqref{eq:A}) and the fact that $\|\mathbf{y}\|=\|\mathbf{x}\|$ we have 
\begin{align*}
    A(\mathbf x,\mathbf y,\eta)^2=\|\mathbf{x}\|^2+\|\mathbf{y}\|^2-2\langle \mathbf y,\eta\rangle=2\|\mathbf{y}\|^2-2\langle \mathbf y,\eta\rangle.
\end{align*}

We first prove the upper bound in~\eqref{eq:precise_measure}. Observe that $(\| \mathbf x\|^2+\|\mathbf y\|^2-2\langle \mathbf y,\eta\rangle)/4 t \leq 1/2$ for all $\eta\in U(\mathbf y, t)$. Thus, applying~\eqref{heat:Rosler}, we get 
\begin{equation}
     \begin{split}
         \mu_{\mathbf x} (U(\mathbf y, t))=\int_{U(\mathbf y, t)}d\mu_{\mathbf x} (\eta)
         & \leq e^{1/2} \int_{U(\mathbf y, t)}e^{-(\| \mathbf x\|^2+\|\mathbf y\|^2-2\langle \mathbf y,\eta\rangle)/4 t}\, d\mu_{\mathbf x} (\eta)\\
         &\leq  e^{1/2} \int_{\mathbb R^N} e^{-(\| \mathbf x\|^2+\|\mathbf y\|^2-2\langle \mathbf y,\eta\rangle)/4 t}\, d\mu_{\mathbf x}(\eta)\\
        &\leq Ct^{\mathbf N/2} h_{t} (\mathbf x,\mathbf y)\leq C'\frac{t^{\mathbf N/2}\Lambda(\mathbf x,\mathbf y,t)}{w(B(\mathbf x,\sqrt{t}))},
      \end{split}
\end{equation}
where in the last inequality we have  used~\eqref{eq:main_claim}.

We now turn to prove the lower bound in~\eqref{eq:precise_measure}. 
From Theorem~\ref{teo:1}, the fact that $d(\mathbf{x},\mathbf{y})=0$,~\eqref{eq:Lambda_2t}, and the doubling property~\eqref{eq:doubling}, we deduce that there is a constant $C>0$ being independent of $\mathbf x$, $\sigma\in G$, and $t>0$  such that 
\begin{align*}
    h_{2t}(\mathbf x,\sigma(\mathbf x))\leq C_{u}w(B(\mathbf{x},\sqrt{2t}))^{-1}\Lambda(\mathbf{x},\sigma(\mathbf{x}),2t) \leq C_{u}'w(B(\mathbf{x},\sqrt{t}))^{-1}\Lambda(\mathbf{x},\sigma(\mathbf{x}),t) \leq C h_t(\mathbf x,\mathbf \sigma (\mathbf x)).
\end{align*} 
Hence, using~\eqref{heat:Rosler} applied to $h_{2t}(\mathbf x,\sigma(\mathbf x))$ and $h_t(\mathbf x,\mathbf \sigma (\mathbf x))$ together with~\eqref{eq:Lambda_2t} and the doubling property~\eqref{eq:doubling}, we conclude that there is a constant $\widetilde{c}\geq 0$ independent of $\mathbf{x}$, $\sigma \in G$, and $t>0$ such that 
\begin{equation}\label{eq:int_mux}
\int_{\mathbb{R}^N} e^{-(2\|\mathbf{y}\|^2-2\langle \mathbf{y},\eta\rangle)/8t}\, d\mu_{\mathbf x}(\eta)\leq e^{\widetilde{c}}\int_{\mathbb{R}^N} e^{-(2\|\mathbf{y}\|^2-2\langle \mathbf{y}, \eta\rangle )/4t}\, d\mu_{\mathbf x}(\eta).
\end{equation}
Let $M=4(\widetilde{c}+1)$. We rewrite \eqref{eq:int_mux} by splitting the areas of the integration: 
\begin{equation}\label{eq:IIJJ}\begin{split}
I_U+I_V:&=\int_{U(\mathbf y,Mt)}e^{-(\|\mathbf{y}\|^2-\langle \mathbf y,\eta\rangle) /4t}\, d\mu_{\mathbf x}(\eta)+ \int_{V(\mathbf y,Mt)}e^{-(\|\mathbf{y}\|^2-\langle \mathbf y,\eta\rangle) /4t}\, d\mu_{\mathbf x}(\eta) \\
&\leq e^{\widetilde{c}}\int_{U(\mathbf y,Mt)} e^{-(\|\mathbf{y}\|^2-\langle \mathbf{y}, \eta\rangle )/2t}\, d\mu_{\mathbf x}(\eta) +e^{\widetilde{c}}\int_{V(\mathbf y,Mt)} e^{-(\|\mathbf{y}\|^2-\langle \mathbf{y}, \eta\rangle )/2t}\, d\mu_{\mathbf x}(\eta)=:J_U+J_V. 
\end{split}\end{equation}
Observe that, by the definition of $V(\mathbf{y},Mt)$ (see~\eqref{eq:V}), and the fact that $M=4(\widetilde{c}+1)$, for all $\eta \in V(\mathbf{y},Mt)$ we have
\begin{equation}\label{eq:calc}
    \begin{split}
   \frac{1}{2}e^{-(\|\mathbf{y}\|^2-\langle \mathbf y,\eta\rangle )/4t}&=\frac{1}{2}e^{-(\|\mathbf{y}\|^2-\langle \mathbf y,\eta\rangle )/2t}e^{(\|\mathbf{y}\|^2-\langle \mathbf y,\eta\rangle )/4t} \geq \frac{1}{2}e^{-(\|\mathbf{y}\|^2-\langle \mathbf y,\eta\rangle )/2t}e^{M/4}\\&=\frac{1}{2}e^{-(\|\mathbf{y}\|^2-\langle \mathbf y,\eta\rangle )/2t}e^{\widetilde{c}+1} \geq e^{\widetilde{c}}e^{-(\|\mathbf{y}\|^2-\langle \mathbf y,\eta\rangle )/2t}.
   \end{split}
\end{equation}
Consequently,~\eqref{eq:calc} implies $J_V\leq \frac{1}{2}I_V$. Therefore, by~\eqref{eq:IIJJ}, we get
\begin{align*}
    I_U+\frac{1}{2} I_V \leq J_U.
\end{align*}
Applying Theorem~\ref{teo:1}, we deduce that 
\begin{equation}
    \begin{split}
   \frac{C_l}{2} \frac{\Lambda(\mathbf x,\mathbf y, 2t)}{w(B(\mathbf x,\sqrt{2t}))}&\leq \frac{1}{2} h_{2t}(\mathbf x,\mathbf y)=\frac{1}{2} {\boldsymbol c}_k^{-1} (4t)^{-\mathbf N/2} (I_U+I_V)\leq  {\boldsymbol c}_k^{-1} (4t)^{-\mathbf N/2} J_U\\
   & \leq {\boldsymbol c}_k^{-1} (4t)^{-\mathbf N/2} e^{\widetilde{c}} \mu_{\mathbf x} (U(\mathbf y, Mt) ). 
    \end{split}
\end{equation}
Now the claim follows from~\eqref{eq:Lambda_2t} and the doubling property \eqref{eq:doubling}.  
\end{proof}

}

We want to remark that Theorem \ref{teo:rejeb} extends the  result of Jiu and Li \cite[Theorem 2.1]{JiuLi}, where the behavior of $\mu_{\mathbf x}$ around  $\mathbf x$ is studied. 

\ 

{\bf Acknowledgment.}  The authors want to thank Jean-Philippe Anker for drawing the author's attention to some references.

\end{document}